\documentclass[11pt,twoside,a4paper,reqno]{amsart}

\usepackage{amsmath}
\usepackage{amsthm}
\usepackage{amsfonts, amssymb}
\usepackage[all]{xy}
\usepackage{url}

\usepackage{latexsym}

\usepackage{graphicx}
\usepackage{graphics}
\usepackage[scriptsize,bf]{caption}
\usepackage{float}
\usepackage{enumerate}
\usepackage{verbatim}

\theoremstyle{plain}
\newtheorem{lema}{Lemma}[section]
\newtheorem{prop}[lema]{Proposition}
\newtheorem{teo}[lema]{Theorem}

\newtheorem{coro}[lema]{Corollary}

\theoremstyle{definition}
\newtheorem*{defi}{Definition}

\newtheorem{obs}[lema]{Remark}

\newcommand{\bd}{\partial}

\begin{document}

\title[The non-pure version of $\Delta^d$ and $\partial\Delta^d$]{The non-pure version of the simplex and the boundary of the simplex}

\author[N.A. Capitelli]{Nicol\'as A. Capitelli}

\subjclass[2010]{55M05, 52B70, 57Q99}

\keywords{Simplicial complexes, combinatorial manifolds, Alexander dual}

\address{Departamento de Matem\'atica-IMAS\\
 FCEyN, Universidad de Buenos Aires\\ Buenos
Aires, Argentina.}

\email{ncapitel@dm.uba.ar} 

\begin{abstract} We introduce the non-pure versions of simplicial balls and spheres with minimum number of vertices. These are a special type of non-homogeneous balls and spheres ($NH$-balls and $NH$-spheres) satisfying a minimality condition on the number of maximal simplices. The main result is that \emph{minimal} $NH$-balls and $NH$-spheres are precisely the simplicial complexes whose iterated Alexander duals converge respectively to a simplex or the boundary of a simplex.
\end{abstract}

\maketitle

\section{Introduction}

A simplicial complex $K$ of dimension $d$ is \emph{vertex-minimal} if it is a simplex or it has $d+2$ vertices. It is not hard to see that a vertex-minimal homogeneous (or pure) complex of dimension $d$ is either an elementary starring $(\tau,a)\Delta^d$ of a $d$-simplex or the boundary $\partial\Delta^{d+1}$ of a ($d+1$)-simplex. On the other hand, a general non-pure complex with minimum number of vertices has no precise characterization. However, since vertex-minimal pure complexes are either balls or spheres, it is natural to ask whether there is a non-pure analogue to these polyhedra within the theory of non-homogeneous balls and spheres. $NH$-balls and $NH$-spheres are the non-necessarily pure versions of combinatorial balls and spheres. They are part of a general theory of non-homogeneous manifolds ($NH$-manifolds) recently introduced by G. Minian and the author \cite{CM}. The study of $NH$-manifolds was in part motivated by Bj\"orner and Wachs's notion of non-pure shellability \cite{BjWa} and by their relationship with factorizations of Pachner moves between (classical) manifolds. $NH$-balls and $NH$-spheres share many of the basic properties of combinatorial balls and spheres and they play an equivalent role to these in the generalized non-pure versions of classical manifold theorems. In a recent work \cite{CM2}, the results of Dong and Santos-Sturmfels on the homotopy type of the Alexander dual of simplicial balls and spheres were generalized to the non-homogeneous setting: the Alexander dual of an $NH$-ball is a contractible space and the Alexander dual of an $NH$-sphere is homotopy equivalent to a sphere (see \cite{Dong,SS}). It was also shown in \cite{CM2} that non-homogeneous balls and spheres are the \emph{Alexander double duals} of classical balls and spheres. This result establishes a natural connection between the pure and non-pure theories.

The purpose of this article is to introduce \emph{minimal} $NH$-balls and $NH$-spheres, which are respectively the non-pure versions of vertex-minimal balls and spheres. Note that $\partial\Delta^{d+1}$ is not only the $d$-sphere with minimum number of vertices but also the one with minimum number of maximal simplices. For non-pure spheres, this last property is strictly stronger than vertex-minimality and it is convenient to define minimal $NH$-spheres as the ones with minimum number of maximal simplices. With this definition, minimal $NH$-spheres with the homotopy type of a $k$-sphere are precisely the non-pure spheres whose nerve is $\partial\Delta^{k+1}$, a property that also characterizes the boundary of simplices. On the other hand, an $NH$-ball $B$ is minimal if it is part of a decomposition of a minimal $NH$-sphere, i.e. if there exists a combinatorial ball $L$ with $B\cap L=\partial L$ such that $B+L$ is a minimal $NH$-sphere. This definition is consistent with the notion of vertex-minimal simplicial ball (see Lemma \ref{Lema:equivalencesOfMinimalBalls} below).

Surprisingly, minimal $NH$-balls and $NH$-spheres can be characterized independently of their definition by a property involving Alexander duals. Denote by $K^*$ the Alexander dual of a complex $K$ relative to the vertices of $K$. Put inductively $K^{*(0)}=K$ and $K^{*(m)}=(K^{*(m-1)})^*$. Thus, in each step $K^{*(i)}$ is computed relative to its own vertices, i.e. as a subcomplex of the sphere of minimum dimension containing it. We call $\{K^{*(m)}\}_{m\in\mathbb{N}_0}$ the \emph{sequence of iterated Alexander duals} of $K$. The main result of the article is the following

\begin{teo} \label{Teo: ppal} \mbox{}
\begin{itemize}
\item[($i$)] There is an $m\in\mathbb{N}_0$ such that $K^{*(m)}=\partial\Delta^d$ if and only if $K$ is a minimal $NH$-sphere.
\item[($ii$)]
There is an $m\in\mathbb{N}_0$ such that $K^{*(m)}=\Delta^d$ if and only if $K$ is a minimal $NH$-ball.
\end{itemize}
\end{teo}

Note that $K^*=\Delta^d$ if and only if $K$ is a vertex-minimal $d$-ball which is not a simplex, so ($ii$) describes precisely all complexes converging to vertex-minimal balls. Theorem \ref{Teo: ppal} characterizes the classes of $\Delta^d$ and $\partial\Delta^d$ in the equivalence relation generated by $K\sim K^*$.

\section{Preliminaries}

\subsection{Notations and definitions}

All simplicial complexes that we deal with are assumed to be finite. Given a set of vertices $V$ , $|V|$ will denote its cardinality and $\Delta(V)$ the simplex spanned by its vertices. $\Delta^d = \Delta(\{0,\ldots,d\})$ will denote a generic $d$-simplex and $\partial\Delta^d$ its boundary. The set of vertices of a complex $K$ will be denoted $V_K$ and we set $\Delta_K := \Delta(V_K)$. A simplex is \emph{maximal} or \emph{principal} in a complex $K$ if it is not a proper face of any other simplex of $K$. We denote by $\mathsf{m}(K)$ the number of principal simplices in $K$. A \emph{ridge} is a maximal proper face of a principal simplex. A complex is \emph{pure} or \emph{homogeneous} if all its maximal simplices have the same dimension.

$\sigma\ast\tau$ will denote the join of the simplices $\sigma$ and $\tau$ (with $V_{\sigma}\cap V_{\tau}=\emptyset$) and $K\ast L$ the join of the complexes $K$ and $L$ (where $V_K\cap V_L=\emptyset$). By convention, if $\emptyset$ is the empty simplex and $\{\emptyset\}$ the complex containing only the empty simplex then $K\ast\{\emptyset\}=K$ and $K\ast\emptyset=\emptyset$. Note that $\partial\Delta^0=\{\emptyset\}$. For $\sigma\in K$,  $lk(\sigma,K)=\{\tau\in K:\ \tau\cap\sigma=\emptyset,\ \tau\ast\sigma\in K\}$ denotes its \emph{link} and $st(\sigma,K)=\sigma\ast lk(\sigma,K)$ its \emph{star}. The union of two complexes $K, L$ will be denoted by $K+L$. A subcomplex $L\subset K$ is said to be \emph{top generated} if every principal simplex of $L$ is also principal in $K$.

$K\searrow L$ will mean that $K$ (simplicially) collapses to $L$. A complex is \emph{collapsible} if it has a subdivision which collapses to a single vertex. The \emph{simplicial nerve} $\mathcal{N}(K)$ of $K$ is the complex whose vertices are the principal simplices of $K$ and whose simplices are the finite subsets of principal simplices of $K$ with non-empty intersection.

Two complexes are \emph{$PL$-isomorphic} if they have a common subdivision. A \textit{combinatorial $d$-ball} is a complex $PL$-isomorphic to $\Delta^d$. A \textit{combinatorial $d$-sphere} is a complex $PL$-isomorphic to $\partial\Delta^{d+1}$. By convention, $\partial\Delta^0=\{\emptyset\}$ is a sphere of dimension $-1$. A \textit{combinatorial $d$-manifold} is a complex $M$ such that $lk(v,M)$ is a combinatorial ($d-1$)-ball or ($d-1$)-sphere for every $v\in V_M$. A ($d-1$)-simplex in a combinatorial $d$-manifold $M$ is a face of at most two $d$-simplices of $M$ and the boundary $\partial M$ is the complex generated by the ($d-1$)-simplices which are face of exactly one $d$-simplex. Combinatorial $d$-balls and $d$-spheres are combinatorial $d$-manifolds. The boundary of a combinatorial $d$-ball is a combinatorial ($d-1$)-sphere.

\subsection{Non-homogeneous balls and spheres}\label{Subsection:NHBallsAndSpheres}

In order to make the presentation self-contained, we recall first the definition and some basic properties of non-homogeneous balls and spheres. For a comprehensive exposition of the subject, the reader is referred to \cite{CM} (see also \cite[$\S 2.3$]{CM2} for a brief summary).

$NH$-balls and $NH$-spheres are special types of $NH$-manifolds, which are the non-necessarily pure versions of combinatorial manifolds. $NH$-manifolds have a local structure consisting of regularly-assembled pieces of Euclidean spaces of different dimensions. In Figure 1 we show some examples of $NH$-manifolds and their underlying spaces. $NH$-manifolds, $NH$-balls and $NH$-spheres are defined as follows.

\begin{defi} An \emph{$NH$-manifold} (resp. \emph{$NH$-ball}, \emph{$NH$-sphere}) of dimension $0$ is a manifold (resp. ball, sphere) of dimension $0$. An $NH$-sphere of dimension $-1$ is, by convention, the complex $\{\emptyset\}$. For $d\geq 1$, we define by induction

 \begin{itemize}
 \item An \emph{$NH$-manifold} of dimension $d$ is a complex $M$ of dimension $d$ such that $lk(v,M)$ is an $NH$-ball of dimension $0\leq k\leq d-1$ or an $NH$-sphere of dimension $-1\leq k\leq d-1$ for all $v\in V_M$.
 \item An \emph{$NH$-ball} of dimension $d$ is a collapsible $NH$-manifold of dimension $d$.
 \item An \emph{$NH$-sphere} of dimension $d$ and \emph{homotopy dimension} $k$ is an $NH$-manifold $S$ of dimension $d$ such that there exist a top generated $NH$-ball $B$ of dimension $d$ and a top generated combinatorial $k$-ball $L$ such that $B + L=S$ and $B\cap L=\bd{L}$. We say that  $S=B+L$ is a \emph{decomposition} of $S$ and write $\dim_h(S)$ for the homotopy dimension of $S$.\end{itemize}\end{defi}
 
\begin{figure}[h]
\centering
\includegraphics[width=6.00in,height=2.00in]{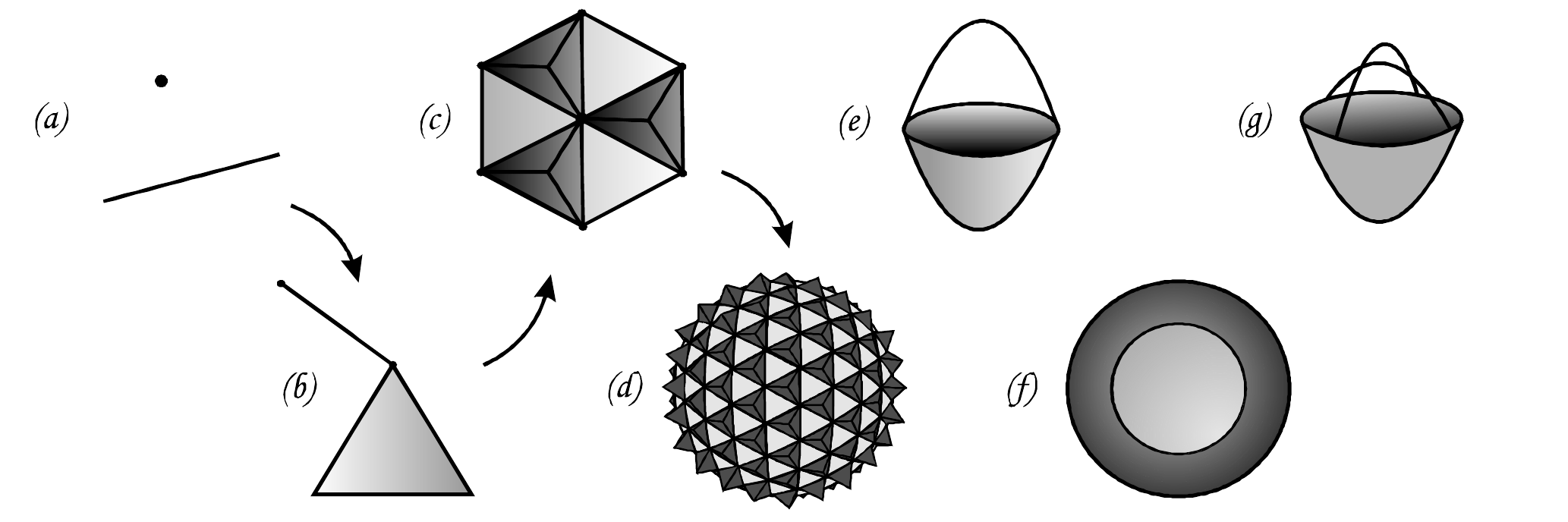}
\caption{Examples of $NH$-manifolds. ($a$), ($d$) and ($e$) are $NH$-spheres of dimension $1$, $3$ and $2$ and homotopy dimension $0$, $2$ and $1$ respectively. ($b$) is an $NH$-ball of dimension $2$ and ($c$), ($f$) are $NH$-balls of dimension $3$. ($g$) is an $NH$-manifold which is neither an $NH$-ball nor an $NH$-sphere. The sequence ($a$)-($d$) evidences how $NH$-manifolds are inductively defined.}
\end{figure}

The definitions of $NH$-ball and $NH$-sphere are motivated by the classical theorems of Whitehead and Newman (see e.g.  \cite[Corollaries 3.28 and 3.13]{RoSa}). Just like for classical combinatorial manifolds, it can be seen that the class of $NH$-manifolds (resp. $NH$-balls, $NH$-spheres) is closed under subdivision and that the link of \emph{every} simplex in an $NH$-manifold is an $NH$-ball or an $NH$-sphere. Also, the homogeneous $NH$-manifolds (resp. $NH$-balls, $NH$-spheres) are precisely the combinatorial manifolds (resp. balls, spheres). Globally, a connected $NH$-manifold $M$ is (non-pure) \emph{strongly connected}: given two principal simplices $\sigma,\tau\in M$ there is a sequence of maximal simplices $\sigma=\eta_1,\ldots,\eta_t=\tau$ such that $\eta_i\cap \eta_{i+1}$ is a ridge of $\eta_i$ or $\eta_{i+1}$ for every $1\leq i\leq t-1$ (see \cite[Lemma 3.15]{CM}). In particular, $NH$-balls and $NH$-spheres of homotopy dimension greater that $0$ are strongly connected.

Unlike for classical spheres, non-pure $NH$-spheres do have boundary simplices; that is, simplices whose links are $NH$-balls. However, for any decomposition $S=B+L$ of an $NH$-sphere and any $\sigma\in L$, $lk(\sigma,S)$ is an $NH$-sphere with decomposition $lk(\sigma,S)=lk(\sigma,B)+lk(\sigma,L)$ (see \cite[Lemma 4.8]{CM}). In particular, if $\sigma\in B\cap L$ then $lk(\sigma,B)$ is an $NH$-ball.

\subsection{The Alexander dual.}\label{Subsection:AlexanderDual} For a finite simplicial complex $K$ and a ground set of vertices $V\supseteq V_K$, the \emph{Alexander dual} of $K$ (relative to $V$) is the complex $$K^{*_V}=\{\sigma\in\Delta(V)\,|\,\Delta(V-V_{\sigma})\notin K\}.$$ The main importance of $K^{*_V}$ lies in the combinatorial formulation of Alexander duality: $H_i(K^{*_V})\simeq H^{n-i-3}(K)$. Here $n=|V|$ and the homology and cohomology groups are reduced (see e.g. \cite{BjTa}). In what follows, we shall write $K^*:=K^{*_{V_K}}$ and $K^{\tau}:=K^{*_V}$ if $\tau=\Delta(V-V_K)$. With this convention, $K^{\tau}=K^*$ if $\tau=\emptyset$. Note that $(\Delta^d)^*=\emptyset$ and $(\partial\Delta^{d+1})^*=\{\emptyset\}$.

The relationship between Alexander duals relative to different ground sets of vertices is given by the following formula (see \cite[Lemma 3.1]{CM2}): \begin{equation}\label{eq:unica} K^{\tau}=\partial\tau\ast\Delta_K+\tau\ast K^*.\tag{$\ast$}\end{equation} Here $K^*$ is viewed as a subcomplex of $\Delta_K$. It is easy to see from the definition that $(K^*)^{\Delta(V_K-V_{K^*})}=K$ and that $(K^{\tau})^*=K$ if $K\neq\Delta^d$ (see \cite[Lemma 3.1]{CM2}). The following result characterizes the Alexander dual of vertex-minimal complexes.

\begin{lema}[{\cite[Lemma 4.1]{CM2}}] \label{lema:dualvertexminimal} If $K=\Delta^d+u\ast lk(u,K)$ with $u\notin\Delta^d$, then $K^*=lk(u,K)^{\tau}$ where $\tau=\Delta(V_K-V_{st(u,K)})$.\end{lema}

It can be shown that $K^{\tau}$ is an $NH$-ball (resp. $NH$-sphere) if and only if $K^*$ is an $NH$-ball (resp. $NH$-sphere). This actually follows from the next result involving a slightly more general form of formula \eqref{eq:unica}, which we include here for future reference.

\begin{lema}[{\cite[Lemma 3.5]{CM2}}] \label{lema:technical} If $V_K\subset V$ and $\eta\neq\emptyset$, then $L:=\partial\eta\ast\Delta(V)+\eta\ast K$ is an $NH$-ball (resp. $NH$-sphere) if and only if $K$ is an $NH$-ball (resp. $NH$-sphere).\end{lema}

\section{Minimal $NH$-spheres}

In this section we introduce the non-pure version of $\partial\Delta^d$ and prove part ($i$) of Theorem \ref{Teo: ppal}. Recall that $\mathsf{m}(K)$ denotes the number of maximal simplices of $K$. We shall see that for a non-homogeneous sphere $S$, requesting minimality of $\mathsf{m}(S)$ is strictly stronger than requesting that of $V_S$. This is the reason why vertex-minimal $NH$-spheres are not necessarily \emph{minimal} in our sense.

To introduce minimal $NH$-spheres we note first that any complex $K$ with the homotopy type of a $k$-sphere has at least $k+2$ principal simplices. This follows from the fact that the simplicial nerve $\mathcal{N}(K)$ is homotopy equivalent to $K$.

\begin{defi} An $NH$-sphere $S$ is said to be \emph{minimal} if $\mathsf{m}(S)=\dim_h(S)+2$.\end{defi}

Note that, equivalently, an $NH$-sphere $S$ of homotopy dimension $k$ is minimal if and only if $\mathcal{N}(S)=\partial\Delta^{k+1}$.

\begin{obs} \label{lemalinkminimales} Suppose $S=B+L$ is a decomposition of a minimal $NH$-sphere of homotopy dimension $k$ and let $v\in V_L$. Then $lk(v,S)$ is an $NH$-sphere of homotopy dimension $\dim_h(lk(v,S))=k-1$ and $lk(v,S)=lk(v,B)+lk(v,L)$ is a valid decomposition (see $\S$2.2). In particular, $\mathsf{m}(lk(v,S))\geq k+1$. Also, $\mathsf{m}(lk(v,S))<k+3$ since $\mathsf{m}(S)<k+3$ and $\mathsf{m}(lk(v,S))\neq k+2$ since otherwise $S$ is a cone. Therefore, $\mathsf{m}(lk(v,S))=k+1=\dim_h(lk(v,S))+2$, which shows that $lk(v,S)$ is also a minimal $NH$-sphere.\end{obs}

We next prove that minimal $NH$-spheres are vertex-minimal.

\begin{prop}\label{lemmavertices} If $S$ is a $d$-dimensional minimal $NH$-sphere then $|V_S|=d+2$.\end{prop}

\begin{proof} Let $S=B+L$ be decomposition of $S$ and set $k=\dim_h(S)$. We shall prove that $|V_S|\leq d+2$ by induction on $k$. The case $k=0$ is straightforward, so assume $k\geq 1$. Let $\eta\in B$ be a principal simplex of minimal dimension and let $\Omega$ denote the intersection of all principal simplices of $S$ different from $\eta$. Note that $\Omega\neq\emptyset$ since $\mathcal{N}(S)=\partial\Delta^{k+1}$ and let $u\in\Omega$ be a vertex. Since $\eta\notin L$ then $\Omega\subset L$ and $u\in L$. By Remark \ref{lemalinkminimales}, $lk(u,S)$ is a minimal $NH$-sphere of dimension $d'\leq d-1$ and homotopy dimension $k-1$. By inductive hypothesis, $|V_{lk(u,S)}|\leq d'+2\leq d+1$. Hence, $st(u,S)$ is a top generated subcomplex of $S$ with $k+1$ principal simplices and at most $d+2$ vertices. By construction, $S=st(u,S)+\eta$. We claim that $V_{\eta}\subset V_{st(u,S)}$. Since $B=st(u,B)+\eta$, by strong connectivity there is a ridge $\sigma\in B$ in $st(u,B)\cap\eta$ (see $\S$2.2). By the minimality of $\eta$ we must have $\eta=w\ast\sigma$ for some vertex $w$. Now, $\sigma\in st(u,B)\cap\eta \subset st(u,S)\cap\eta$; but $st(v,S)\cap \eta\neq\sigma$ since, otherwise, $S=st(u,S)+\eta\searrow st(u,S)\searrow u$, contradicting the fact that $S$ has the homotopy type of a sphere. We conclude that $w\in st(u,S)$ since every face of $\eta$ different from $\sigma$ contains $w$.  Thus, $|V_S|=|V_{st(u,S)}\cup V_{\eta}|=|V_{st(u,S)}|\leq d+2$.\end{proof}

This last proposition shows that, in the non-pure setting, requesting the minimality of $\mathsf{m}(S)$ is strictly more restrictive than requesting that of $|V_S|$. For example, a vertex-minimal $NH$-sphere can be constructed from \emph{any} $NH$-sphere $S$ and a vertex $u\notin S$ by the formula $\tilde{S}:=\Delta_S+u\ast S$. It is easy to see that if $S$ is not minimal, neither is $\tilde{S}$.

\begin{obs} \label{u is in L spheres} By Proposition \ref{lemmavertices}, a $d$-dimensional minimal $NH$-sphere $S$ may be written $S=\Delta^d+u\ast lk(u,S)$ for some $u\notin\Delta^d$. Note that for any decomposition $S=B+L$, the vertex $u$ must lie in $L$ (since this last complex is top generated). In particular, $lk(u,S)$ is a minimal $NH$-sphere by Remark \ref{lemalinkminimales}.\end{obs}

As we mentioned above, the Alexander duals play a key role in characterizing minimal $NH$-spheres. We now turn to prove Theorem \ref{Teo: ppal} ($i$). We derive first the following corollary of Proposition \ref{lemmavertices}.

\begin{coro}\label{corollarydropsvertices} If $S$ is a minimal $NH$-sphere then $|V_{S^*}| < |V_S|$ and $\dim(S^*)<\dim(S)$.\end{coro}

\begin{proof} $V_{S^*}\subsetneq V_S$ follows from Proposition \ref{lemmavertices} since if $S=\Delta^d+u\ast lk(u,S)$ then $u\notin S^*$. In particular, this implies that $\dim(S^*)\neq\dim(S)$ since $S^*$ is not a simplex by Alexander duality.\end{proof}

\begin{teo}	\label{corollarydualitypreservesminimality}

Let $K$ be a finite simplicial complex and let $\tau$ be a simplex (possibly empty) disjoint from $K$. Then, $K$ is a minimal $NH$-sphere if and only if $K^{\tau}$ is a minimal $NH$-sphere. That is, the class of minimal $NH$-spheres is closed under taking Alexander dual.

\begin{proof} Assume first that $K$ is a minimal $NH$-sphere and set $d=\dim(K)$. We proceed by induction on $d$. By Proposition \ref{lemmavertices}, we can write $K=\Delta^d+u\ast lk(u,K)$ for $u\notin\Delta^d$. If $\tau=\emptyset$ then, by Lemma \ref{lema:dualvertexminimal}, $K^*=lk(u,K)^{\rho}$ for $\rho=\Delta(V_K-V_{st(u,K)})$. By Remark \ref{u is in L spheres}, $lk(u,K)$ is a minimal $NH$-sphere. Therefore, $K^*=lk(u,K)^{\rho}$ is a minimal $NH$-sphere by inductive hypothesis. If $\tau\neq\emptyset$, $K^{\tau} = \partial\tau\ast\Delta_K+\tau\ast K^*$ by formula \eqref{eq:unica}. In particular, $K^{\tau}$ is an $NH$-sphere by Lemma \ref{lema:technical} and the case $\tau=\emptyset$.  Now, by Alexander duality,
$$\dim_h(K^{\tau})=|V_K\cup V_{\tau}|-\dim_h(K)-3=|V_K|+|V_{\tau}|-\dim_h(K)-3=\dim_h(K^*)+|V_{\tau}|.$$ On the other hand, $$\mathsf{m}(K^{\tau})=\mathsf{m}(\partial\tau\ast\Delta_K+\tau\ast K^*)=\mathsf{m}(\partial\tau)+\mathsf{m}(K^*)=|V_{\tau}|+\dim_h(K^*)+2,$$ where the last equality follows from the case $\tau=\emptyset$. This shows that $S^{\tau}$ is minimal.

Assume now that $K^{\tau}$ is a minimal $NH$-sphere. If $\tau\neq\emptyset$ then $K=(K^{\tau})^*$ and if $\tau=\emptyset$ then $K=(K^*)^{\Delta(V_K-V_{K^*})}$ (see \S 2.3). In any case, the result follows immediately from the previous implication.\end{proof}

\end{teo}

\begin{proof}[Proof of Theorem \ref{Teo: ppal} \normalfont\text{($i$)}] Suppose first that $K$ is a minimal $NH$-sphere. By Theorem \ref{corollarydualitypreservesminimality}, every non-empty complex in the sequence $\{K^{*(m)}\}_{m\in\mathbb{N}_0}$ is a minimal $NH$-sphere. By Corollary \ref{corollarydropsvertices}, $|V_{K^{*(m+1)}}|<|V_{K^{*(m)}}|$ for all $m$ such that $K^{*(m)}\neq\{\emptyset\}$. Therefore, $K^{*(m_0)}=\{\emptyset\}$ for some $m_0<|V_K|$ and hence $K^{*(m_0-1)}=\partial\Delta^d$ for some $d\geq 1$.

Assume now that $K^{*(m)}=\partial\Delta^d$ for some $m\in\mathbb{N}_0$ and $d\geq 1$. We proceed by induction on $m$. The case $m=0$ corresponds to the trivial case $K=\partial\Delta^d$. For $m\geq 1$, the result follows immediately from Theorem \ref{corollarydualitypreservesminimality} and the inductive hypothesis.\end{proof}

\section{Minimal $NH$-balls}

We now develop the notion of minimal $NH$-ball. The definition in this case is a little less straightforward that in the case of spheres because there is no piecewise-linear-equivalence argument in the construction of non-pure balls. To motivate the definition of minimal $NH$-ball, recall that for a non-empty simplex $\tau\in K$ and a vertex $a\notin K$, the \emph{elementary starring} $(\tau,a)$ of $K$ is the operation which transforms $K$ in $(\tau,a)K$ by removing $\tau\ast lk(\tau,K)=st(\tau,K)$ and replacing it with $a\ast\partial\tau\ast lk(\tau,K)$. Note that when $\dim(\tau)=0$ then $(\tau,a)K$ is isomorphic to $K$.

\begin{lema}\label{Lema:equivalencesOfMinimalBalls}

	Let $B$ be a combinatorial $d$-ball. The following statements are equivalent.\begin{enumerate}
	\item $|V_B|\leq d+2$ (i.e. $B$ is vertex-minimal).
	\item $B$ is an elementary starring of $\Delta^d$.
	\item There is a combinatorial $d$-ball $L$ such that $B+L=\partial\Delta^{d+1}$.
	\end{enumerate}
	
	\begin{proof} We first prove that (1) implies (2) by induction on $d$. Since $\Delta^d$ is trivially a starring of any of its vertices, we may assume $|V_B|=d+2$ and write $B=\Delta^d+u\ast lk(u,B)$ for $u\notin\Delta^d$. Since $lk(u,B)$ is necessarily a vertex-minimal ($d-1$)-combinatorial ball then $lk(u,B)=(\tau,a)\Delta^{d-1}$ by inductive hypothesis. It follows from an easy computation that $B$ is isomorphic to $(u\ast\tau,a)\Delta^d$.
	
	We next prove that (2) implies (3). We have $$B=(\tau,a)\Delta^d=a\ast\partial\tau\ast lk(\tau,\Delta^d)=a\ast\partial\tau\ast\Delta^{d-\dim(\tau)-1}=\partial\tau\ast\Delta^{d-\dim(\tau)}.$$ Letting $L:=\tau\ast\partial\Delta^{d-\dim(\tau)}$ we get the statement of (3).
	
	The other implication is trivial.\end{proof}
\end{lema}

\begin{defi} An $NH$-ball $B$ is said to be \emph{minimal} if there exists a minimal $NH$-sphere $S$ that admits a decomposition $S=B+L$.\end{defi}
	
Note that if $B$ is a minimal $NH$-ball and $S=B+L$ is a decomposition of a minimal $NH$-sphere then, by Remark \ref{lemalinkminimales}, $lk(v,B)$ is a minimal $NH$-ball for every $v\in B\cap L$ (see $\S$\ref{Subsection:NHBallsAndSpheres}). Note also that the intersection of all the principal simplices of $B$ is non-empty since $\mathcal{N}(B)\subsetneq\mathcal{N}(S)=\partial\Delta^{k+1}$. Therefore, $\mathcal{N}(B)$ is a simplex. The converse, however, is easily seen to be false.

The proof of Theorem \ref{Teo: ppal} ($ii$) will follow the same lines as its version for $NH$-spheres.

\begin{prop}\label{lemmavertices2} If $B$ is a $d$-dimensional minimal $NH$-ball then $|V_B|\leq d+2$.\end{prop}

\begin{proof} This follows immediately from Proposition \ref{lemmavertices} since $\dim(B)=\dim(S)$ for any decomposition $S=B+L$ of an $NH$-sphere.\end{proof}

\begin{coro}\label{corollarydropsvertices2} If $B$ is a minimal $NH$-ball then $|V_{B^*}| < |V_B|$ and $\dim(B^*)<\dim(B)$.\end{coro}

\begin{proof} We may assume $B\neq\Delta^d$. $V_{B^*}\subsetneq V_B$ by the same reasoning made in the proof of Corollary \ref{corollarydropsvertices}. Also, if $\dim(B)=\dim(B^*)$ then $B^*=\Delta^d$. By formula \eqref{eq:unica},  $B=(B^*)^{\rho}=\partial\rho\ast\Delta^d$ where $\rho=\Delta(V_B-V_{B^*})$, which is a contradiction since $|V_B|=d+2$.\end{proof}

\begin{obs}\label{u is in L balls} The same construction that we made for minimal $NH$-spheres shows that vertex-minimal $NH$-balls need not be minimal. Also, similarly to the case of non-pure spheres, if $B=\Delta^d+u\ast lk(u,B)$ is a minimal $NH$-ball which is not a simplex then for any decomposition $S=B+L$ of a minimal $NH$-sphere we have $u\in L$. In particular, since $lk(u,S)=lk(u,B)+lk(u,L)$ is a valid decomposition of a minimal $NH$-sphere, then
$lk(u,B)$ is a minimal $NH$-ball (see Remark \ref{u is in L spheres}).\end{obs}

\begin{teo} \label{corollarydualitypreservesminimality2} Let $K$ be a finite simplicial complex and let $\tau$ be a simplex (possibly empty) disjoint from $K$. Then, $K$ is a minimal $NH$-ball if and only if $K^{\tau}$ is a minimal $NH$-ball. That is, the class of minimal $NH$-balls is closed under taking Alexander dual.\end{teo}
	
\begin{proof} Assume first that $K$ is a minimal $NH$-ball and proceed by induction on $d=\dim(K)$. The case $\tau=\emptyset$ follows the same reasoning as the proof of Theorem \ref{corollarydualitypreservesminimality} using the previous remarks. Suppose then $\tau\neq\emptyset$. Since by the previous case $K^*$ is a minimal $NH$-ball, there exists a decomposition $\tilde{S}=K^*+\tilde{L}$ of a minimal $NH$-sphere. By Propositions \ref{lemmavertices} and \ref{lemmavertices2}, either $K^*$ is a simplex (and $V_{\tilde{S}}-V_{K^*}=\{w\}$ is a single vertex) or $V_{\tilde{S}}=V_{K^*}\subset V_K$. Let $S:=K^{\tau}+\tau\ast\tilde{L}$, where we identify the vertex $w$ with any vertex in $V_K-V_{K^*}$ if $K^*$ is a simplex. We claim that $S=K^{\tau}+\tau\ast\tilde{L}$ is a valid decomposition of a minimal $NH$-sphere. On one hand, formula \eqref{eq:unica} and Lemma \ref{lema:technical} imply that $K^{\tau}$ is an $NH$-ball and that $$S=\partial\tau\ast\Delta_K+\tau\ast K^*+\tau\ast\tilde{L}=\partial\tau\ast\Delta_{K}+\tau\ast\tilde{S}$$ is an $NH$-sphere. Also,
\begin{align*}
K^{\tau}\cap(\tau\ast\tilde{L})&=(\partial\tau\ast\Delta_{K}+\tau\ast K^*)\cap(\tau\ast\tilde{L})\\
&=\partial\tau\ast\tilde{L}+\tau\ast(K^*\cap\tilde{L})\\
&=\partial\tau\ast\tilde{L}+\tau\ast\partial\tilde{L}\\
&=\partial(\tau\ast\tilde{L}).
\end{align*} This shows that $S=K^{\tau}+\tau\ast\tilde{L}$ is valid decomposition of an $NH$-sphere. On the other hand, $$\mathsf{m}(S)=
\mathsf{m}(\partial\tau)+\mathsf{m}(\tilde{S})=\dim(\tau)+1+\dim(\tilde{L})+2=\dim_h(S)+2,$$ which proves that $S$ is minimal. This settles the implication.

The other implication is analogous to the corresponding part of the proof of Theorem \ref{corollarydualitypreservesminimality}.\end{proof}	

\begin{proof}[Proof of Theorem \ref{Teo: ppal} \normalfont($ii$)] It follows the same reasoning as the proof of Theorem \ref{Teo: ppal} ($i$) (replacing $\{\emptyset\}$ with $\emptyset$).\end{proof}

If $K^*=\Delta^d$ then, letting $\tau=\Delta(V_K-V_{\Delta^d})\neq\emptyset$, we have $K=(K^*)^{\tau}=\partial\tau\ast\Delta^d=(\tau,v)\Delta^{d+\dim(\tau)}$. This shows that Theorem \ref{Teo: ppal} ($ii$)  characterizes all complexes which converge to vertex-minimal balls.

\section{Further properties of minimal $NH$-balls and $NH$-spheres}

In this final section we briefly discuss some characteristic properties of minimal $NH$-balls and $NH$-spheres.

\begin{prop} \label{Prop:links_are_minimal} In a minimal $NH$-ball or $NH$-sphere, the link of every simplex is a minimal $NH$-ball or $NH$-sphere.\end{prop}

\begin{proof} Let $K$ be a minimal $NH$-ball or $NH$-sphere of dimension $d$ and let $\sigma\in K$. We may assume $K\neq\Delta^d$. Since for a non-trivial decomposition $\sigma=w\ast\eta$ we have $lk(\sigma,S)=lk(w,lk(\eta,S))$, by an inductive argument it suffices to prove the case $\sigma=v\in V_K$. We proceed by induction on $d$. We may assume $d\geq 1$. Write $K=\Delta^d+u\ast lk(u,K)$ where, as shown before, $lk(u,K)$ is either a minimal $NH$-ball or a minimal $NH$-sphere. Note that this in particular settles the case $v=u$. Suppose then $v\neq u$. If $v\notin lk(u,K)$ then $lk(v,K)=\Delta^{d-1}$. Otherwise, $lk(v,K)=\Delta^{d-1}+u\ast lk(v,lk(u,K))$. By inductive hypothesis, $lk(v,lk(u,K))$ is a minimal $NH$-ball or $NH$-sphere. By Lemma \ref{lema:dualvertexminimal}, $$lk(v,K)^*=lk(v,lk(u,K))^{\rho},$$ and the result follows from Theorems \ref{corollarydualitypreservesminimality} and \ref{corollarydualitypreservesminimality2}.\end{proof}

For any vertex $v\in K$, the \emph{deletion} $K-v=\{\sigma\in K\,|\, v\notin\sigma\}$ is again a minimal $NH$-ball or $NH$-sphere. This follows from Proposition \ref{Prop:links_are_minimal}, Theorems \ref{corollarydualitypreservesminimality} and \ref{corollarydualitypreservesminimality2} and the fact that $lk(v,K^*)=(K-v)^*$ for any $v\in V_K$ (see \cite[Lemma 4.2 (1)]{CM2}). Also, Remark \ref{u is in L balls} implies that minimal $NH$-balls are (non-pure) vertex-decomposable as defined by Bj\"orner and Wachs (see \cite[\S 11]{BjWa2}).

Finally, we make use of Theorems \ref{corollarydualitypreservesminimality} and \ref{corollarydualitypreservesminimality2} to compute the number of minimal $NH$-spheres and $NH$-balls in each dimension.

\begin{prop}

	Let $0\leq k\leq d$.\begin{enumerate}
	\item There are exactly $\binom{d}{k}$ minimal $NH$-spheres of dimension $d$ and homotopy dimension $k$. In particular, there are exactly $2^d$ minimal $NH$-spheres of dimension $d$.
	\item There are exactly $2^d$ minimal $NH$-balls of dimension $d$.
	\end{enumerate}
	
	\begin{proof} We first prove ($1$). An $NH$-sphere with $d=k$ is homogeneous by \cite[Proposition 2.4]{CM2}, in which case the result is obvious. Assume then $0\leq k\leq d-1$ and proceed by induction on $d$. Let $\mathcal{S}_{d,k}$ denote the set of minimal $NH$-spheres of dimension $d$ and homotopy dimension $k$. If $S\in\mathcal{S}_{d,k}$ it follows from Theorem \ref{corollarydualitypreservesminimality}, Corollary \ref{corollarydropsvertices} and Alexander duality that $S^*$ is a minimal $NH$-sphere with $\dim(S^*)<d$ and $\dim_h(S^*)=d-k-1$. Therefore, there is a well defined application
	$$\mathcal{S}_{d,k}\stackrel{f}{\longrightarrow}\bigcup_{i=d-k-1}^{d-1}\mathcal{S}_{i,d-k-1}$$ sending $S$ to $S^*$. We claim that $f$ is a bijection. To prove injectivity, suppose $S_1,S_2\in\mathcal{S}_{d,k}$ are such that $S_1^*=S_2^*$. Let $\rho_i=\Delta(V_{S_i}-V_{S_i^*})$ ($i=1,2$). Since $|V_{S_1}|=d+2=|V_{S_2}|$ then $\dim(\rho_1)=\dim(\rho_2)$ and, hence, $S_1=(S_1^*)^{\rho_1}=(S_2^*)^{\rho_2}=S_2$. To prove surjectivity, let $\tilde{S}\in\mathcal{S}_{j,d-k-1}$ with $d-k-1\leq j\leq d-1$. Taking $\tau=\Delta^{d-j-1}$ we have $\tilde{S}^{\tau}\in\mathcal{S}_{d,k}$ and $f(\tilde{S}^{\tau})=\tilde{S}$ (see $\S$2.3). Finally, using the inductive hypothesis, $$|\mathcal{S}_{d,k}|=\sum_{i=d-k-1}^{d-1}|\mathcal{S}_{i,d-k-1}|=\sum_{i=d-k-1}^{d-1}\binom{i}{d-k-1}=\binom{d}{k}.$$
	
	For (2), let $\mathcal{B}_d$ denote the set of minimal $NH$-balls of dimension $d$ and proceed again by induction on $d$. The very same reasoning as above gives a well defined bijection $$\mathcal{B}_d-\{\Delta^d\}\stackrel{f}{\longrightarrow}\bigcup_{i=0}^{d-1}\mathcal{B}_i.$$ Therefore, using the inductive hypothesis, $$|\mathcal{B}_d-\{\Delta^d\}|=\sum_{i=0}^{d-1}|\mathcal{B}_i|=\sum_{i=0}^{d-1}2^i=2^d-1.\qedhere$$\end{proof}\end{prop}

\subsection*{Acknowledgement} I am grateful to Gabriel Minian for many helpful remarks and suggestions during the preparation of the paper.

\end{document}